\documentclass{amsart}
\usepackage{pdfsync}
\usepackage{microtype}
\usepackage{libertinus} 
\usepackage[capitalise]{cleveref} 
\usepackage{cite} 
\usepackage{tikz} 
    \usetikzlibrary{backgrounds,shadows.blur}
\usepackage{adjustbox}
\usepackage{floatflt} 
\usepackage[margin=8pt,font=footnotesize,labelfont=bf,labelsep=endash]{caption} 
\newtheorem{theorem}{Theorem}
\newtheorem{lemma}[theorem]{Lemma}

\newtheorem{define}[theorem]{Definition}
\newtheorem{proposition}[theorem]{Proposition}
\newtheorem{corollary}[theorem]{Corollary}
\newtheorem{definition}[theorem]{Definition}

\newtheorem{remark}[theorem]{Remark}
\newtheorem{question}[theorem]{Question}
\numberwithin{theorem}{section}

\DeclareMathOperator{\Extg}{Ext_{\Gamma}^{-1}}
\DeclareMathOperator{\Ind}{Ind}

\newcommand{\C}{\mathbb{C}} 
\newcommand{\complex}{\C}
\newcommand{\Cv}{\ensuremath{C_\upsilon  (M)}}
\newcommand{\closure}[2][3]{
        {}\mkern#1mu\overline{\mkern-#1mu#2}}
\newcommand{\compact}{\ensuremath{\mathcal K}}
\newcommand{\compacts}{\compact}
\newcommand{\Cstar}[1][]{\textsl{C*}{#1}} 

\newcommand{\E}{\mathcal E} 
\newcommand{\Id}{\text{Id}}
\newcommand{\idealpoint}{\ensuremath{\{\inf\}}}
\newcommand{\inv}{^{-1}}
\newcommand{\kasparovproduct}{\tensor} 
\newcommand{\KE}{\compact(E)} 
\newcommand{\KK}{{\textsl{KK}}} 
\newcommand{\KKg}{\ensuremath{{\KK_\Gamma}}}

\newcommand{\adjointables}{\mathcal L}
\newcommand{\LE}{\adjointables (E)} 

\newcommand{\Mbar}{\ensuremath{\closure{M}}}

\newcommand{\Mult}{\mathcal{M}}

\newcommand{\R}{\mathbb{R}} 
\newcommand{\tensor}{\otimes}
\renewcommand{\H}[1]{{\ensuremath{\mathbb H^{#1}} }}  
\renewcommand{\inf}{\infty} 
\renewcommand{\L}[1]{{\ensuremath{{ L}^{#1}} }} 
\renewcommand{\S}[1]{{\ensuremath{S^{#1}} }}  
\renewcommand{\star}{\phantom{}^{*}}
\newcommand{\Z}[1][]{\mathbb Z_{#1}} 
\usepackage{mathtools}
\DeclarePairedDelimiterX\ip[2]\langle\rangle{
	\ifblank{#1}{\ifblank{#2}{\:\cdot\:,\:\cdot\:}{\:\cdot\:,#2}}{\ifblank{#2}{#1,\:\cdot\:}{#1,#2}}
}

\newcommand{\CurvedTrianglePic}{
\begin{tikzpicture}[auto,   background rectangle/.style={fill=gray!27},show background rectangle,every shadow/.style={shadow opacity=23,
shadow blur steps=9}]]
\node (ideal_infty) at (0,2) [circle,fill,inner sep=0.7] {};
\node (n1) at (-0.7,-0.12) [circle,fill,inner sep=0.8]  {};
\node (n2) at (0.7,-0.12) [circle,fill,inner sep=0.8]  {};
\node (left_corner) at (-1.5,0) {};
\node (right_corner) at (1.5,0) {};
\shadedraw [blur shadow={shadow yshift=-3ex}] (ideal_infty)  to [bend left,very thin]  (left_corner)
  to [bend right=5 ]  (n1) 
  to [bend right=5 ]  (n2)
  to [bend right=5 ]  (right_corner)
  to [bend left, very thin]  (ideal_infty);                                
\draw[densely dotted,in=-90,out=60]  (n1) to  (ideal_infty); 
\draw[densely dotted,in=-90,out=120]   (n2)  to    (ideal_infty); 
\draw (ideal_infty) node[above] {$\scriptscriptstyle \idealpoint$};
\node at (0,0.6) {$\scriptstyle M$}; 
\node at (n1) [below ]{$\scriptstyle n_1$}; 
\node at (n2) [below ]{$\scriptstyle n_2$};
\end{tikzpicture}      }
\begin{document}
\title{A short proof of an index theorem, II}
\author{Y. Abdolmaleki and D. Kucerovsky}

\maketitle \begin{quotation}\ {\small {\bf Abstract:} We introduce a slight modification of the usual equivariant $KK$-theory.
   We use this to give a $KK$-theoretical proof of
  an equivariant index theorem for Dirac-Schr\"{o}dinger operators on a non-compact manifold of nowhere positive curvature. We incidentally show that the boundary of Dirac is Dirac; generalizing earlier work of Baum and coworkers, and a result of Higson and Roe. }
\end{quotation}
\section{Introduction}

In this paper we define a form of approximate equivariance in $\KKg$-theory (\cref{def:Gadmissible.approx,prop:approx.equiv.phi}), and apply it to make the proof of Anghel's theorem in \cite{shortproof} equivariant. Some exposition has been added, in order to enhance the readability of this paper and to motivate and explain the techniques used. The main technical tools are presented in \cref{sect:approx.equivariance,sect:virtual.spaces}. The purpose of these tools is to assist in modifying non-equivariant proofs in $\KK$-theory to the equivariant case, and the non-equivariant proof that we use as a case study is that of Anghel's theorem in \cite{shortproof}. The tools are demonstrated in two slightly different settings:   hyperbolic space  with $SO(n)$ acting by isometries, and a Hadamard manifold with a discrete amenable group acting by isometries (\cref{sect:Hadamard,sect:hyperbolic.space}).  The main application is, as already mentioned, to Anghel's theorem.

 We begin with an extended introduction of and discussion of Anghel's index theorem.
Index theorems, generally speaking, express an analytical index in
terms of topological information. An analytical index is usually some
generalization of the classical Fredholm index of a Fredholm operator,
and the relevant topological information is usually given by (the
cohomological image of) a $K$-theory $/$ $K$-homology pairing. One example is the Atiyah-Singer index
theorem \cite{ATSIT}, and another is in  the cases where the
Baum-Connes conjecture \cite{BCH} holds.

Anghel's index theorem is an index formula for the Fredholm index of a Dirac-Schr\"{o}dinger
operator. These operators are of the form $D+iV,$ where $D$ is a
(generalized) Dirac operator and $iV$ is a skew-adjoint order zero
operator, acting on the complex $\L2$-sections of some bundle.
The reason for being interested in an equivariant form of Anghel's theorem  is that an equivariant index is much more sensitive than the usual Fredholm index. The Fredholm index is integer-valued, while the equivariant index takes values in the representation ring of a group.

The classic Anghel's theorem was first proven by Callias \cite{callias} for the case of Euclidean space, and
then in the case of warped cones by Anghel \cite{anghel} :

\begin{theorem}[\bf Anghel's theorem]\label{th:original.Anghel's}
  Let $D+iV$ be a Dirac-Schr\"{o}dinger
  operator over a warped cone $M$ with compact even-dimensional base
  $N.$ If $V^2$ becomes arbitrarily large outside a compact subset of $
    M,$
  and $[D,V]$ is bounded, then $D+iV$ is Fredholm, with index given by
  \[\int_N\hat {A}(TN)\wedge\mbox{\rm ch}V^{+}d\mbox{(\rm vol}_N),\]
  where $\hat {A}$ denotes Atiyah's $A$-genus, and $V^{+}$ is the positive eigenbundle
  of $V$ over a copy of $N$ contained in a neighborhood of infinity such that $
    V$ is invertible in that neighborhood.
\end{theorem}
The above theorem
applies to \emph{warped cones.} These are manifolds which
are isomorphic outside a compact set to $\R^{+}\times N$ with Riemannian
metric $dr^2+f(r)^2\tilde {g},$ where $\tilde {g}$ is the Riemannian metric of the compact
manifold $N,$ and $f$ is a nondecreasing function $f:{\R^+}\to
  {\R}^{+}.$ We remark that a hyperbolic space $H^{n+1}$ is naturally a warped cone, and locally a warped product, of the form $[0,\infty)\times S^{n}$ with metric $dr^2+\sinh^2(r) \tilde{g}.$
Anghel's original proof of his theorem \cite{anghel} used differential geometry, and a $\KK$-theoretical proof, which we will generalize in our work here, can be found in \cite{shortproof}.

\section{Formulation of a Anghel-type theorem}
Our aim  is to first prove that a certain natural exact sequence maps the class of a Dirac operator in equivariant $K$-homology to the class of a Dirac operator, and thus to prove a simple version of an equivariant Anghel-type theorem. However, one  has to formulate the problem appropriately because otherwise the problem may either lack applications or lack a solution.

Suppose that a manifold $M$ is a warped cone, as defined just above, whose collar is of the form $(1,\infty)\times N$ for some compact manifold $N.$ We can define a \Cstar-algebra $\Cv$ consisting of those continuous bounded functions on $M$ whose radial limits at infinity exist uniformly.
In \cite{shortproof} there is a proof of an Anghel-type theorem, in the non-equivariant case, based on the cyclic exact sequences induced in $KK$-theory by the semisplit exact sequence
\begin{equation} 0\to C_0(M)\to \Cv\to C(N)\to 0.\label{ses:Cv}\tag{2}\end{equation}
In this short exact sequence the quotient map is the operation of taking the limit as the radial variable $r$ goes to infinity.
The middle term can be viewed as continuous functions on a compactification, $\Mbar,$ of $M.$ This particular compactification, $\Mbar,$ is a compact topological manifold with boundary. Since we want to have a discrete group acting by isometries on the manifolds $M$ and $N,$ it is reasonable to first look at the case where $M$ is a space of constant curvature, see \cite{doCarmo} for more information on these.
We will thus first focus on the case of hyperbolic space for simplicity in exposition, but then give a complete proof for the general case of a warped cone of nowhere positive curvature.

We will show the following two facts:
\begin{enumerate}\item \textbf{(The boundary of Dirac is Dirac)} The short exact sequence \eqref{ses:Cv} maps the $K$-homology cycle associated with a Dirac operator $D_M$ on $M$ to a  $K$-homology cycle associated with a Dirac operator $D_N$ on $N,$ and\label{item:bdd}
  \item \textbf{(Equivariant Anghel)} for suitable classes of approximately equivariant potentials $V$ over $M$ we have $[D_M +iV]=[f]\kasparovproduct_{C(N)}[D_N] ,$ where $f$ is a $K$-theory element over $N$ and the equality is in the representation ring  $\KKg(\C,\C).$
\end{enumerate}
The proof of the first fact will involve an interesting construction of an equivariant $\KK$-group that is only approximately equivariant but defines the same group as the usual definition. We remark that the term used above, that the boundary of Dirac is Dirac, originates with  Baum, Douglas, and Taylor \cite[Sect. 4.5]{BDT}. They construct a geometrical $K$-homology theory in which it is a fundamental fact that property \ref{item:bdd} above holds, and they show this for both the non-equivariant case and then for the case of equivariance under a compact Lie group \cite{BDT,baum,higsonBoundary}.

The next two sections give some preparatory constructions, and we return to the above topics in \cref{sect:boundary.of.Dirac}.

\section{Approximate equivariance in $\KKg (A,B)$}\label{sect:approx.equivariance}

Generally, the new issues that may arise when adding a group action to a non-equivariant $\KK$-theoretical proof have to do with the existence of cycles.
Thus we begin with some remarks on approximately equivariant cycles.

Following \cite[Chapter 8]{kucerovsky.thesis},
recall that a locally compact group $\Gamma$ is said to have an action
on a graded \Cstar--algebra $B$ if there is a homomorphism $\alpha$ from $
  \Gamma$ into
the degree zero $*$--automorphisms of $B.$ 
The group acts on a Hilbert
$B$--module $\E$ by a homomorphism, also denoted $\alpha ,$ into the
invertible bounded linear transformations on
$\E$ as a Banach space, with $\alpha_g(eb)=\alpha_g(e)\alpha_g(b),$
$\alpha_g\langle x,y\rangle =\langle\alpha_gx,\alpha_gy\rangle .$ The space of Hilbert module operators is denoted $\adjointables(E)$ and the ideal of operators that are compact in the Hilbert module sense is denoted $\compacts(E).$
Given a representation $\phi$
of a \Cstar--algebra $A$ on a Hilbert module $E,$  denote by
$I_{\phi}$ the algebra of operators on $E$ that commute with $\phi$
modulo compact operators, and by $J_{\phi}$
the algebra of operators $L\in \LE$ with $\phi (a)L$ and $L\phi
  (a)$
compact for all $a\in A.$
We recall Kasparov's well-known definition of the set of bounded equivariant Kasparov cycles, $
  {\E}_\Gamma(A,B),$ for
$\sigma$--unital \Cstar--algebras $A$ and $B$ and a second--countable locally
compact group $\Gamma$ :
\begin{define}\label{def:KKG}\cite{kasp3} The set
  ${\E}_\Gamma(A,B)$ is given by triples $(\E,\phi ,F)$ such that:
  \begin{enumerate}\item[$i)$]\ $\E$ is a countably generated $\Z[2]$--graded
    Hilbert $B$--module with a continuous degree zero action of\, $\Gamma$;
    \item[$ii)$]\ the map $\phi$ is a graded $*$--homomorphism
    $\phi :A\to \LE$;
    \item[$iii)$]\ there is an action of $g$ on $A$ such that
    $\alpha_g(\phi (a))=\phi (\alpha_g(a))$;
    \item[$iv)$]\ the operator $F$ is such that $g\mapsto\phi (a)\alpha_g(F)$ is norm--continuous;
    \item[$v)$]\ the degree one operator $F$ is in $I_{\phi},$ and is such
    that $F^2-1,$ $F-F\star,$ and $\alpha_g(F)-F$ are in $J_{\phi}.$
  \end{enumerate}\end{define}

The following Theorem shows that the condition $iii$ can be replaced by the weaker condition:
\begin{enumerate}\item[$iii')$]\textit{ there is an action of $g$ on $A$ such that
  $\alpha_g(\phi (a)) - \phi (\alpha_g(a))\in \KE.$}\label{eqn:weaker.equivariance}
\end{enumerate}

\begin{theorem} For a discrete amenable group $\Gamma,$ we obtain the same Kasparov group $\KKg(A,B)$ from either of the sets of conditions: $i,ii,iii,iv,$ and $v$ \emph{or}  $i,ii,iii',iv,$ and $v.$ \label{prop:approx.equiv.phi}\end{theorem}
\begin{proof}
  Recall that there is a group of $\Gamma$-extensions, usually denoted  $\Extg(A,B).$ It consists of equivariant Busby maps, modulo an equivalence relation, see \cite{thomsen} for more information. These extensions are required to be equivariantly semisplit, as will be explained.
  There is a well-known isomorphism of $\KKg^1 (A,B)$ onto $\Extg(A,B),$ of the form $(E,\phi,F)\mapsto\pi(F\phi(\cdot)F\star),$ where $\pi$ is the natural quotient map $\pi\colon \LE\to \LE/\KE.$ This isomorphism takes a Fredholm module to a Busby map, and
  we could choose to define the equivalence relations on $\KKg^1 (A,B)$ through pulling back the equivalence relations
  from $\Extg(A,B).$ If we take this point of view, then it is clear that the difference between the two conditions on Fredholm modules:
  \textit{
        \begin{enumerate}\item[$iii')$]\ there is an action of $g$ on $A$ such that
                                $\alpha_g(\phi (a)) - \phi (\alpha_g(a))\in \KE$
                                          {\rm and}
                          \item[$iii)$]\ there is an action of $g$ on $A$ such that
                                $\alpha_g(\phi (a)) = \phi (\alpha_g(a))$
        \end{enumerate}
        }
  disappears under the equivalence relations, because in either case we have
  $$ \pi(F(\phi(\alpha_g(a)))F\star)=\pi(F(\alpha_g(\phi(a)))F\star) .$$

  It remains to check the often delicate semisplitting property of the extension that we obtain under the above isomorphism, while using the weaker condition $iii'.$ We must verify that the extension obtained has an equivariant semisplitting. The usual candidate for such a semisplitting is the map  $a\mapsto F\phi(a)F\star,$ and with  the condition $iii',$ we can immediately only conclude that we have a semisplitting, $a\mapsto F\phi(a)F\star.$  Baaj and Skandalis have however shown that in the discrete and amenable case, the existence of a semisplitting implies the existence of an equivariant semisplitting , see \cite[prop. 7.13(2)]{BS},  and  \cite[prop. 7.16]{BS}.) 
\end{proof}

For the reader's convenience we recall our earlier definition of equivariant unbounded cycles in \cite[Def. 8.7]{kucerovsky.thesis}, and \cite[Def. 4.7, p. 272]{kucerovsky.kaspcon4}.

\begin{define}\label{def:Gadmissible}The set of unbounded
  equivariant Kasparov
  modules $\Psi_\Gamma (A,B)$ is given by triples $(E,\phi ,D)$ where $E$ is a
  Hilbert $B$--module with $\Gamma$-action; $\phi :A\to {\mathcal L}
    (E)$ is a
  $*$--homomorphism; and $D$ is an
  unbounded regular degree one self--adjoint operator on $E$, such
  that:
  \begin{enumerate}\item[$i)$]  for each $e\in E$, the map \(g\mapsto\left\{D-\alpha_g(D)\right\}e\)
      is continuous as a map from $\Gamma$ into $E$;
      \item[$ii)$]\ the operator $(i+D)\inv$ is in $J_{\phi}$;
    \item[$iii)$] the homomorphism $\phi$ satisfies $\alpha_g(\phi (a)e)=\phi (\alpha_g(a))e$; and
      \item[$iv)$]\ for all $a$ in some dense subalgebra of $A,$ the
      commutator $[D,\phi (a)]$ is bounded on the domain of $D$.
  \end{enumerate}\end{define}
Pointwise norm $\Gamma$-continuity (strong
continuity in the Hilbert module sense) is all that is needed in part $
  i$ of the above definition. This definition behaves correctly under the bounded transform, and the
unbounded connection conditions for a Kasparov product still hold and have the same form in the equivariant case as in the non-equivariant case \cite{kucerovsky.kaspcon4}. Because passing to the bounded transform \cite{BJ} only affects the operator $D,$ we  can replace, in \cref{def:Gadmissible} above, the condition $iii$ by the condition $iii'$ (from  page \pageref{eqn:weaker.equivariance})  using  \cref{prop:approx.equiv.phi} applied to the bounded transforms.\par
The following Definition and Corollary summarize our discussion:

\begin{definition} \label{def:Gadmissible.approx}The set of unbounded approximately
  equivariant Fredholm
  modules, $\Psi_\Gamma(A,B),$ is given by triples $(E,\phi ,D)$ where $E$ is a (graded)
  Hilbert $B$--module with $\Gamma$-action; $\phi :A\to\LE$ is a
  $*$--homomorphism; and $D$ is an
  unbounded regular degree one self--adjoint operator on $E$, such
  that:
  \begin{enumerate}\item[$i)$] for each $e\in E$, the map \(g\mapsto\left\{D-\alpha_g(D)\right\}e\)
      is continuous as a map from $\Gamma$ into $E$;
      \item[$ii)$]\ the operator $(i+D)^{-1}$ is in $J_{\phi}$;
    \item[$iii')$] the homomorphism $\phi$ satisfies $\alpha_g(\phi (a))-\phi (\alpha_g(a
          ))\in\KE,$ and
      \item[$iv)$]\ for all $a$ in some dense subalgebra of $A,$ the
      commutator $[D,\phi (a)]$ is bounded on the domain of $D$.
  \end{enumerate}  \end{definition}
  An alternative form of condition $i$ above is given in \cref{def:Gadmissible.approx.locally.bounded.form} in the Appendix.
   The above definition does not require amenability.  In the presence of amenability, \cref{prop:approx.equiv.phi} implies
 \begin{corollary}  In the discrete and amenable case, $\Psi_\Gamma (A,B)$ is isomorphic to the usual $\KKg(A,B)$ group.\end{corollary}
 By assuming or omitting invertibility one can avoid the amenability condition.

\begin{remark} If three cycles satisfy the usual connection conditions when the group action is forgotten, then they also form a Kasparov product in the equivariant case.\end{remark}

\section{A delocalized $\L{2}$-space and amenablity}\label{sect:virtual.spaces}
In this section we construct some convenient Hilbert spaces that will be needed in the proofs of our theorems (specifically, they will be used to define $K$-homology cycles).
Recall the unital exact sequence of $\Cstar$-algebras from the introduction, namely:
\begin{equation*} 0\to C_0(M)\to \Cv\to C(N)\to 0.\end{equation*}
This extension is semisplit by the Choi-Effros theorem. If we now equip this extension with a group action that is compatible with the inclusion map, quotient map, and semisplitting map, it will induce \cite[Theorem 7.17]{BS}  cyclic exact sequences in equivariant $KK$-theory. Fortunately, it is not necessary to explicitly verify compatibility of the group action with the semisplitting:   by Section 7 of \cite{BS}, see also \cite{thomsen.preprint}, we can average an non-equivariant semisplitting, thus obtaining an equivariant semisplitting. Due to the expander-based counterexample in \cite[Section 7]{CounterExample}, some mild condition on the extension or on the group (for example, amenability) seems to be needed for such a semisplitting. For an example of an explicit equivariant semisplitting, please see the proof of Theorem 3.4 in \cite{mesland2020}.


  \begin{remark}[\textbf{On Dirac operators}]\label{rem:Dirac.ops}{\rm  The topological manifold with boundary $\Mbar$ naturally inherits a metric from $M,$ and this metric is singular at the boundary.  Bismut and Cheeger \cite{BismutCheeger} constructed the natural spinor Dirac operator $D_{\Mbar}$ on a spin$^c$ manifold with boundary and a radially singular metric at the boundary. Their operator is singular at the boundary. The  spin bundle associated to this singular operator $D_{\Mbar}$ restricts to a  spin bundle, say $E,$ over $M,$ and since Dirac operators do not increase support, we obtain a convenient Dirac operator $D_M$ on $M.$ Bismut and Cheeger construct at the same time a related Dirac operator $D_N$ on the boundary. These operators are given by local formulas of the expected form:
\begin{eqnarray*}
  D_{\Mbar} &=& \sum_{i=1}^{n} e_i' \nabla_{e_i'}^{\Mbar}  \\
  &\mbox{and}&\\
  D_N &=& \sum_{i=1}^{n-1} e_i \nabla_{e_i}^N.
\end{eqnarray*}
  The above are formulas 1.6 and 1.7 from \cite{BismutCheeger}. For more information, please see \cite[pgs. 319-324]{BismutCheeger}, or for general information on Dirac operators see \cite{GilbertMurray}.
} \end{remark}
To define cycles in $K$-homology we need $L^2$-spaces. Let us use the notation $L^2 (M,E)$ for an $L^2$ space with coefficients in a vector bundle $E.$ The following Proposition will be applied to the spinor bundle associated with a Dirac operator, but can as well be stated in a more general form.
\newcommand\source[1]{%
    \tikz[remember picture,baseline,inner sep=0pt] {%
        \node [name=source,anchor=base]{$#1$};
    }%
    \setcounter{target}{0}
}

\newcounter{target}
\newcommand\target[1]{%
    \tikz[remember picture,baseline,inner sep=0pt] {%
        \node [name=target-\thetarget,anchor=base]{$#1$};
    }%
    \stepcounter{target}%
}

\newcommand\drawarrows{
    \tikz[remember picture, overlay, bend left=40, densely dotted,<-] {
        \foreach \i [evaluate=\i as \n using int(\i-1)] in {1,...,\thetarget} {
            \draw (source.north) to (target-\n.north);
        }
    }
}

\begin{proposition}\label{prop:equivariantL2}Consider a semisplit short exact sequence
\begin{center}\begin{tikzpicture}[auto,/tikz/xscale=0.86]
   \node (0l) at (-2.8,0)      {$0$};
    \node (Mo) at (-1,0)  {$C_0(M)$};
    \node (M) at  (1,0)     {$C(\Mbar)$};
    \node (N) at  (3,0)     {$C(N)$};
    \node (0r) at (4.8,0)      {$0$};
    \draw[->] (M) to node [swap] {$i$} (N);
    \draw[->] (Mo) to node [swap] {}    (M);
    \draw[->] (0l) to node [swap] {}      (Mo);
    \draw[->] (N) to node [swap] {}       (0r);
    \draw[->,densely dotted,bend right,thick] (N) to node [swap] {$s$} (M);
  \end{tikzpicture}\end{center}
  equipped with the action of a discrete group $\Gamma.$
Suppose that either the group is amenable \underline{\emph{or}} that there exists a $\Gamma$-invariant measure on the  space  $\Mbar.$
Given any vector bundle $E$ over the finite-dimensional manifold with boundary $\Mbar,$
 there exist equivariant $\L{2}$-spaces over $\Mbar$ and over $N$ that are compatible with the equivariant quotient map  $i$
  and its equivariant semisplitting map $s$ with

  \begin{center}\begin{tikzpicture}[auto]
    \node (M) at  (0,0) [left] {$\L{2}(\Mbar,E)$};
    \node (N) at  (2,0) [right] {$\L{2}(N,i\star (E)).$};
    \draw[->] (M) to node [swap] {$i\star$} (N);
    \draw[->,densely dotted,bend right,thick] (N) to node [swap] {$s$} (M);
  \end{tikzpicture}\end{center}

  The $\L2$-space map $s$ is injective, the $\L2$-space map  $i\star$ is  surjective, and $i\star\circ s=\Id.$ There exists also a natural $\L{2}$-subspace $\L{2}(M,E)\subset \L{2}(\Mbar,E).$ \end{proposition}
\begin{proof}Suppose we are in the amenable case. Since the middle term  $C(\Mbar)$ of the given short exact sequence
  \begin{equation*} 0\to C_0(M)\to C(\Mbar)\to C(N)\to 0\end{equation*}
  is a unital \Cstar-algebra, it has a weak-* compact state space. Amenability then implies the existence of a $\Gamma$-invariant state, $\phi,$ in that state space. In terms of measures rather than states, this implies the existence of a nontrivial invariant measure on the Gelfand spectrum $\Mbar$ of this \Cstar-algebra.
  Letting $E$ be the given bundle on $\Mbar,$ let $\L{2}(\Mbar, E)$ denote the corresponding $\L{2}$ space constructed from the bundle $E$ and the above invariant measure coming from $\phi.$ If we compose the invariant state $\phi$ with the linear, unital,  and
  completely positive map provided by the previously discussed equivariant semisplitting $s\colon C(N)\to \Cv,$ we obtain an invariant state $\phi\circ s$ on $C(N).$ Let $i$ denote the map of Gelfand spaces induced by the quotient map in our exact sequence. Let $i^* (E)$ denote the pullback  of the Dirac bundle on $\Mbar$ to $N.$ Let $\L{2}(N,i\star (E))$ denote the $\L{2}$-space constructed from the bundle $i\star (E)$ and the invariant state $\phi\circ s$ on $C(N).$ (This is a special case of the KGNS construction\cite{kasp1}.)

  The semisplitting map $s\colon C(N)\to \Cv$ is not a homomorphism, but it \textit{is} a linear positive map of ordered Banach spaces. Being completely positive, it induces a linear positive map of the space of sections of trivial bundles over $N.$ By Swan's theorem and the compactness of the manifold $N,$ we can regard the space of sections of $i\star (E)$ as a closed subspace of some trivial bundle over $N,$ and then the map $s$ provides a linear positive isomorphism with a closed subspace of the space of sections of the bundle $E$ over $\Mbar.$   The map $s$ therefore splits the given extension at the level of equivariant $\L{2}$ spaces, because the maps $s$ and $i\star$ are at this level respectively an injective map and a surjective map, and $i\star\circ s=\Id.$
  If we are instead given a $\Gamma$-invariant measure on $\Mbar,$ then we replace states by integrals in the above argument.
\end{proof}
By the above Proposition, the map $s$  embeds $\L{2}(N,i\star (E))$ as an $\L{2}$-subspace in $\L{2}(\Mbar,E),$ and even provides a splitting map at the level of $\L2$-spaces.
We will lighten the notation by writing, for example, $\L{2}(N)\subset \L{2}(\Mbar)$ instead of  $s(\L{2}(N,i\star (E)))\subset \L{2}(\Mbar,E).$
The space $\L{2}(N)$ can be said to be \textit{delocalized,} or \textit{``quantum,''} in the sense that it is not obviously induced by a restriction of $\L{2}(\Mbar)$ to a submanifold of $\Mbar.$


We now turn to applications.
\section{The case of hyperbolic space $\H{n+1}$  }\label{sect:boundary.of.Dirac}\label{sect:hyperbolic.space}
The manifold $M$ will, for the moment, be finite-dimensional hyperbolic space $\H{n+1},$  and the manifold $N$ will be the corresponding  boundary $\S{n}.$ We refer to the boundary points that are in  the bounding hyperplane $\R^{n}$ with respect to the half space model as the real  hyperplane, and we denote the remaining boundary point  of $\H{n+1}$ by $\idealpoint.$\label{def.infinity.point} In classic hyperbolic geometry, these boundary points are known as ideal points, and can also be called limit points. Since the space $C(\Mbar)$ is equivalent to $\Cv,$ in other words, continuous functions on $M$ with a mild smoothness condition imposed at infinity, isometries of  $M$ that preserve this condition  extend to uniform homeomorphisms of  $\Mbar.$ As is well-known --- by Brouwer's theorem --- such a homeomorphism will have a fixed point. In order to insure smoothness, consider the class of homeomorphisms  given by the isometries of $M$ whose extensions are isometries of the real hyperplane into itself, and have a common fixed point on the boundary; geometrically speaking, these are  the elements of the (amenable) real orthogonal group $O(n)$ with their natural action on the half-space model. The fixed point is then exactly the special point defined at the start of this section and  denoted $\idealpoint.$ We  denote by $\Gamma$ some given discrete amenable subgroup of this group: we regard it as acting by isometries on $\S{n}$ and on $\H{n+1}.$ We regard the exact sequence \eqref{ses:Cv} from page \pageref{ses:Cv} as being equipped with an action of the discrete amenable subgroup $\Gamma$ acting by isometries, and moreover we suppose that the isometries preserve orientation so that they will be compatible with Dirac operators. We will then generalize to the case of  warped cones of nowhere positive curvature. It is interesting to note that in the case that such a space has constant curvature,  it is isomorphic to a quotient of hyperbolic or Euclidean space modulo the action of a discrete subgroup of isometries $G$ \thickspace \cite[ch. 8]{doCarmo}. The positive scalar curvature case has obstructions \cite{guo} and will not be considered further.
In \cite[Section B.4]{baum} it is shown that the boundary of Dirac is Dirac, in the case of equivariance with respect to the action of a compact Lie group. (See \cite{BDT,higsonBoundary} for proofs in the non-equivariant case.)  The proof  in \cite{baum} uses the compactness of the Lie group in several places, and thus does not generalize in any obvious way to the case of equivariance under an action by a discrete group.
We aim to prove a similar  result in the case of the action of a discrete amenable group. As already mentioned, we treat first the case of hyperbolic space $M=\H{n+1},$ with $N=\S{n},$ and a discrete amenable subgroup of $SO(n)$ acting. Recall that  hyperbolic space has a natural warped cone structure: namely  $[0,\infty)\times S^{n}$ with metric $dr^2+\sinh^2(r) \tilde{g}.$
In the following theorem, the  cycle $[r]\in\KKg^1 (C(N),C_0 (M))$ has operator given by multiplication by the  radial variable of this warped cone $M.$
The Dirac operators below are as in \cref{rem:Dirac.ops}, and the homomorphism $\phi$ is defined next in \cref{prop:phiForHyperbolicCase}.

\begin{theorem} Let $M$ be hyperbolic space $M=\H{n+1},$ and let $N$ be its boundary. Let $\Gamma$ be a discrete amenable subgroup of $SO(n)$ acting on
  $ 0\to C_0(M)\to \Cv\to C(N)\to 0$
  by oriented isometries. The Dirac operators over $N$ and over $M$ give
  $\KKg$-classes $[D_N]\in \KKg^0(C(N),\C)$ and $[D_M]\in \KKg^1(C_0
    (M),\C),$  and $[D_N]=[r]\kasparovproduct [D_M].$ The cycle $[r]\in\KKg^1 (C(N),C_0 (M))$ is given by $(C_0(M),\phi,r).$
\end{theorem}
The proof will take several Lemmas, and concludes in \cref{cor:The boundary of Dirac is Dirac}.
 We begin by using the geometry of hyperbolic space to construct the  $\Gamma$-invariant homomorphism $\phi$ that is used in the cycle $[r].$ The exact property that is used is the visibility space property.

\begin{proposition} There exists a $*$-homomorphism $\phi\colon C(\S{n}) \to \Mult (C_0 (\H{n+1})),$
  and $$\alpha_g(\phi (f))-\phi (\alpha_g(f))\in C_0(\H{n+1}).$$   \label{prop:phiForHyperbolicCase}  \end{proposition}
\begin{proof}

  In the half-space model for hyperbolic space, the given function $f$ is defined on the boundary points (\emph{i.e.} limit points) of the hyperbolic space. Let $n$ denote a limit point in the real hyperplane. Consider the  geodesic through hyperbolic space from the limit point $n$ to the limit point $\idealpoint$. This limit point is  defined on page \pageref{def.infinity.point}.
  Extend the domain of the given function $f$ to all of $\Mbar\setminus\idealpoint$ by defining it to be constant on all such geodesics. In other words, extend $f$ to a larger domain by making it constant with respect to the $y$ co-ordinate of the half-space model, so that $f(n,y):=f(n).$
  The function we obtain is bounded and continuous at all ordinary points of the hyperbolic space, and thus we can use the function as a multiplier of elements of $C_0(\H{n+1}).$ Let us, for convenience, define $\phi(f)$ to be zero at the limit point $\idealpoint,$ so that we may regard $\phi(f)$ as a bounded function on $\Mbar$ that is continuous except at the limit point $\idealpoint.$ Now we notice that $\alpha_g(\phi (f))$ and $\phi (\alpha_g(f))$ are equal when restricted to points of the boundary, $\S{n}.$ But then $\alpha_g(\phi (f)) - \phi (\alpha_g(f))$ is a function that is continuous at all the points of hyperbolic space and its boundary, except possibly the limit point $\idealpoint,$ and moreover it is zero at all other limit points.  This means that this function is in $C_0(\H{n+1})$ as claimed.
  Finally, we remark that the mapping $\phi\colon C(\S{n}) \to \Mult (C_0 (\H{n+1}))$ that we have defined is evidently an algebraic $*$-homomorphism, and algebraic *-homomorphisms of \Cstar-algebras  are automatically bounded (\emph{i.e.} are continuous *-homomorphisms.)
\end{proof}

\subsection{A Kasparov product}

The next lemma is an version of Lemma 3.1 in \cite{shortproof} for an approximately equivariant cycle.  When defining explicit $KK$-elements, it will be useful to note that the complex Clifford algebra $C_1$ is isomorphic to $\complex\oplus \complex$, with elements of the
form $(e,e)$ having even degree, and elements of the form $(e,-e)$ having odd degree. This is sometimes referred to as the odd grading.   The other common grading is the diagonal / off-diagonal grading (for operators), where operators represented by diagonal $2$-by-$2$ matrices have the even grading, and antidiagonal ones the odd grading. The space of Hilbert module operators is denoted $\adjointables(E)$ and the ideal of operators that are compact in the Hilbert module sense is denoted $\compacts(E).$

\begin{lemma}\ Let $S$ be a Dirac bundle over the nowhere positive curvature warped cone $M.$
  Let $H$ be the  Hilbert space
  $L^2(M,S)\oplus L^2(M,S)$ with diagonal / off diagonal grading.
  Let $D_M$ be the Dirac operator on $S.$
 Let $\phi$ denote the approximately invariant homomorphism of Proposition \ref{prop:phiForHyperbolicCase}, and let $r$ denote multiplication by $r,$ the radial coordinate of the warped cone.
  Then the Kasparov product $[r]\kasparovproduct [D_M]$ equals $[D_M+ir],$ where
  \begin{eqnarray*}
    [D_M]&:=&\left(H,m ,\left(\begin{array}{cc}
        0   & D_M \\
        D_M & 0\end{array}
    \right)\right)              \in KK(C_0(M)\tensor C_1 ,\complex),\\{}
    [r]&:=&\left(C_0(M)\otimes C_1,\phi,\left(\begin{array}{cc}
        r & 0  \\
        0 & -r  \end{array}
    \right)\right)              \in KK(C(N),C_0(M)\otimes C_1),\\{}
  &&\mbox{and}\\{}
  [D_{M}+ir]   &:= &\left(H,\phi,\left(\begin{array}{ccc}
            0      & D_M-ir & \\
            D_M+ir & 0 & \end{array}
    \right)\right)              \in KK(C(N),\complex).\\{}\end{eqnarray*}
  \label{lem:c0pairing.equiv} \end{lemma}
  \begin{proof}We follow   the proof of Lemma 3.1 in \cite{shortproof}. The action of $m :C_0(M)\otimes C_1\to \adjointables (H)$ is given by
    \[\begin{array}{rclrcl}
    m :b\oplus b&\mapsto&\left(\begin{array}{cc}
    b&0\\
    0&b\end{array}
    \right),&m :b\oplus -b&\mapsto&\left(\begin{array}{cc}
    0&-ib\\
    ib&0\end{array}
    \right)\end{array}
    ,\]
    where  elements of $C_0(M)$ act on $\L{2}(M,S)$ as multiplication by functions on $M.$
    A calculation shows that there is a
    Hilbert module isomorphism that identifies $H=L^2(M,S)\oplus L^2(M,S)$ with the inner tensor
    product of $C_0(M)\otimes C_1$ and $H$ over $m.$

     In order to apply the criterion for  an unbounded cycle to be the
    Kasparov product of two given cycles \cite{kucerovsky.Ktheory}, we have to verify a semiboundedness condition  and a
    connection condition. Since these conditions do not explicitly involve the group action, the proof of the non-equivariant case (Lemma 3.1 in \cite{shortproof}) goes through, with $[r]$ in the place of $[A]$ there. We comment that $[r]$ is indeed an approximately equivariant cycle, because $r-\alpha_\gamma (r)$ is bounded for each $\gamma\in\Gamma.$
  \end{proof}

\subsection{A homotopy}\label{sect.A.homotopy}
We now make use of these $\L{2}$-spaces to construct a Kasparov homotopy of $K$-homology cycles. This homotopy is an equivariant version of the one used in Lemma 4.1 of \cite{shortproof}.
The statement of the next lemma uses Proposition \ref{prop:phiForHyperbolicCase}, and this means that the lemma is at present
applicable only to the current case of hyperbolic space $M.$ However,
the restriction will be removed once we prove \cref{prop:phiForSymmetricCase}.
The Dirac bundle $S$ in the next Lemma  can be taken to be as in  \cref{rem:Dirac.ops}.
\begin{lemma}\
  Let $S$ be a Dirac bundle over the nowhere positive curvature warped cone $M.$
  Let $D_M$ be the Dirac operator on $S.$
Let $\phi$ be the approximately invariant homomorphism of Proposition \ref{prop:phiForHyperbolicCase}, and let $r$ denote multiplication by $r,$ the radial coordinate of the warped cone. Then $[r]\kasparovproduct [D_M]=[D_N],$ where

  \begin{eqnarray*}
    [D_M]:=\left(L^2(M)\oplus L^2(M),m ,\left(\begin{array}{cc}
        0   & D_M \\
        D_M & 0\end{array}
    \right)\right)              &\in& KK(C_0(M)\tensor C_1 ,\complex),\\{}
    [r]:=\left(C_0(M)\otimes C_1,\phi,\left(\begin{array}{cc}
        r & 0  \\
        0 & -r  \end{array}
    \right)\right)              &\in& KK(C(N),C_0(M)\otimes C_1),\\{}
  &&\mbox{and}\\{}
  [D_{N}]   := \left(L^2(N)\oplus L^2(N),m,\left(\begin{array}{ccc}
            0      & D_N & \\
            D_N & 0 & \end{array}
    \right)\right)             & \in& KK(C(N),\complex).\\{}\end{eqnarray*}
  \label{pr:c0pairing2.equiv} \end{lemma}

\begin{proof} Lemma \ref{lem:c0pairing.equiv} shows that  $[r]\kasparovproduct [D_M]=[D_M + ir],$  so we need only to show that $[D_M + ir]$ is Kasparov homotopic to $[D_N].$
  Thus we need to show a suitable homotopy of $L^2(M)\oplus L^2(M)$ and $L^2(N)\oplus L^2(N).$
  Proposition \ref{prop:equivariantL2} provides equivariant copies of $\L{2}(N)$ and $\L{2}(M)$ embedded in $\L{2}(\Mbar).$ There is a ``dimension drop'' Hilbert $C([0,1])$ module
  $$E:=\{f\in \L{2}(\Mbar)\otimes C([0,1]) \colon
    f(0)\in\L{2}(N)  ,
    f(1)\in \L{2}(M)  \}.$$
    The module consists of the continuous functions from $[0,1]$ into the space $\L{2}(\Mbar)$ which take endpoint values in the subspaces $\L{2}(N)$ at one endpoint and $\L{2}(M)$ at the other endpoint. Since all spaces used are equivariant, the module is equivariant.
    The Kasparov triple    $(E,\phi,D_M +\lambda ir)$ together with the evaluation map (at $\lambda=0$) denoted by $i\star$ in Proposition \ref{prop:equivariantL2}
  provides a homotopy of $(\L{2}(M),\phi,D_M + ir)$ and $(\L{2}(N),\phi,D_N).$ Note that the Dirac operator $D_M$ can be regarded as a Dirac operator on $\Mbar$ by \cref{rem:Dirac.ops}.
  We thus obtain a cycle of the form $\left(L^2(N)\oplus L^2(N),\phi,\left(\begin{array}{ccc}
            0      & D_N & \\
            D_N & 0 & \end{array}
    \right)\right)              \in KK(C(N),\complex),$ but the homomorphism $\phi$ in fact acts on $L^2(N)$ by ordinary multiplication. Thus we obtain the cycle $[ D_N ]=\left(L^2(N)\oplus L^2(N),m,\left(\begin{array}{ccc}
            0      & D_N & \\
            D_N & 0 & \end{array}
    \right)\right)              \in KK(C(N),\complex)$ as was to be shown.
\end{proof}

The above shows that if we regard the operation of taking the Kasparov product with the above cycle $[r]$ as a $K$-homological boundary map, then the boundary of the Dirac operator on $M$ is a Dirac operator on $N.$ This can be attributed in the non-equivariant case to \cite[Sect. 4.5]{BDT}, see also \cite{higsonBoundary,Khomology}, and for the case of a compact Lie group action to \cite[Lemma 3.8]{baum}.

The proof we have given can be viewed as a version of the original non-equivariant proof, which consisted of a Kasparov product calculation followed by noticing that the Busby map associated with their boundary cycle  coincides with the Busby map of the given extension. In order to find a natural candidate for the boundary cycle, and  to state that cycle in a simple way, we introduced a slight generalization of the usual notion of a Fredholm triple.

Technically speaking, we have yet to relate our boundary cycle $[r]$ to the extension
$0\to C_0(M)\to \Cv\to C(N)\to 0$. However, the expected equality at the level of Busby maps  does remain valid in the equivariant case. This is because the formula for  computing the Busby map remains the same with or without equivariance:

\begin{proposition}  The Busby map associated with the cycle $[r]\in \KKg^1 (C(N),C_0 (M))$ under the isomorphism
  from the proof of Proposition \ref{prop:approx.equiv.phi} coincides with the Busby map associated with the equivariant extension $ 0\to C_0(M)\to \Cv\to C(N)\to 0.$
\label{prop:equal.Busby.maps}\end{proposition}
\begin{proof} Same as in the non-equivariant case \cite{higsonBoundary,Khomology}. \end{proof}

Concluding this part of the proof, we combine \cref{pr:c0pairing2.equiv,prop:equal.Busby.maps} to obtain:

\begin{corollary}[\bf The boundary of Dirac is Dirac]  \label{cor:The boundary of Dirac is Dirac}
  The equivariant $K$-homology cycle $[D_M]$ is mapped to $[D_N]$ under the $K$-homology map induced by the equivariant extension $ 0\to C_0(M)\to \Cv\to C(N)\to 0.$
\end{corollary}
The proof of the above result just supposes an amenable discrete group action, acting on finite dimensional spin$^c$ manifolds.  If amenability or finite dimensionality is not assumed, a proof along our lines runs up against one of the expander-based counterexamples to the Baum-Connes conjecture \cite{CounterExample}. These counterexamples are all based on failures of exactness at the level of $K$-theory. Thus, we may wonder if these counterexamples do in fact explicitly manifest at the level of geometric $K$-homology. More precisely:
\vspace{-14pt}\begin{quote}\begin{question} \textit{Is it true that the boundary of Dirac is Dirac, at the level of equivariant $K$-homology,  for actions of finitely generated groups that do not uniformly embed into Hilbert space, as in \cite[Section 7]{CounterExample} ?} \end{question}\end{quote}
\noindent Incidently, the $K$-theory boundary map associated with the cycle $[r]\in\KKg^1 (C(N),C_0 (M))$ can be written \KKg-theoretically as shown below, where $C_1$ acts on itself in the natural way. See \cite[Theorem 3.4]{mesland2020} for an ingenious alternative formulation in the bounded picture, using a Poisson-type transform in the specific case of hyperbolic space.
\begin{proposition} Let $[r]\in\KKg(C(N)\tensor C_1, C_0 (M))$ denote
  $$\left(C_0 (M)\tensor C_1,\phi\tensor\Id\colon C(N)\tensor C_1\to\adjointables,\left(\begin{array}{cc}
        r & 0  \\
        0 & -r\end{array}
    \right)\right).$$  This cycle will
  map the $K$-theory element
  \begin{flalign*} [A]&:=\left(M_n(C (N))\tensor C_1, 1 , \left(\begin{array}{cc}
        A & 0  \\
        0 & -A \end{array}
    \right)\right)   \in  \KKg(\complex, C(N)\tensor C_1) \\
\intertext{to  the $K$-theory element}
 [A+r] &:=\left(M_n (C_0 (M))\tensor C_1,1,\left(\begin{array}{cc}
        \phi(A)+r &  0     \\
        0    &  -\phi(A)-r\\\end{array}
    \right)\right)\\&\shoveright{\in \KKg(\complex, C_0 (M)).}
  \end{flalign*}
  \label{prop:pairing.of.r.and.a.function}\label{prop:map.on.K0}\end{proposition}
\begin{proof}   The (inner) tensor product of Hilbert modules
  $(M_n(C (N))\tensor C_1) \tensor_{\phi\tensor\Id}  C_0(M)\tensor C_1,$ where $\Id\colon C_1\to\adjointables(C_1)$ denotes the natural action of the complex Clifford algebra $C_1$ on itself, is isomorphic to  $M_n(C_0(M)) \tensor C_1$ regarded as a bimodule
  over $C_0 (M)\tensor C_1$ and over $C(N),$ where $C(N)$ acts on $C_0 (M)$ through the map $\phi.$ With respect to this isomorphism, the
  elementary map $T_x \colon C (M)\tensor C_1 \to M_n(C (N))\tensor C_1 \tensor_{\phi\tensor\Id}  C_0 (M)\tensor C_1$ is $\phi(x),$ regarded
  as a multiplication by a matrix-valued function,  $x\in M_n(C(N));$ \textit{i.e.}, a  map $T_x\colon C_0 (M)\tensor C_1 \to   M_n( C_0 (M))\tensor C_1.$ The rest of the proof is a routine verification of the
  connection conditions.
\end{proof}

The case of a degree shift amounts to a  factor of $C_1 :$
\begin{proposition}Let $[r]\in\KKg(C(N), C_0 (M)\otimes C_1)$ denote
  $$\left(C_0 (M)\tensor C_1,\phi\colon C(N)\to\adjointables,\left(\begin{array}{cc}
        r & 0  \\
        0 & -r\end{array}
    \right)\right).$$  This cycle will
  map the $K$-theory element
  \begin{flalign*} [A]&:=\left(M_n(C (N))\tensor C_1, 1 , \left(\begin{array}{cc}
        A & 0  \\
        0 & -A \end{array}
    \right)\right)   \in  \KKg(\complex, C(N)) \\
\intertext{to  the $K$-theory element}
 [A+r] &:=\left(M_n (C_0 (M))\tensor C_1,1,\left(\begin{array}{cc}
        \phi(A)+r &  0     \\
        0    &  -\phi(A)-r\\\end{array}
    \right)\right)\\&\shoveright{\in \KKg(\complex, C_0 (M)\tensor C_1).}
  \end{flalign*}   \label{prop:map.on.K1} \end{proposition}

  \section{The case of Hadamard manifolds }\label{sect:Hadamard}
We started with the case of hyperbolic space, regarded as a warped cone of constant curvature,  with an action of $SO(n).$
In replacing hyperbolic space by a more general warped cone $M$ of nowhere positive curvature, only one new consideration arises, and this has to do with
defining the approximately equivariant $*$-homomorphism $\phi\colon C(N) \to \Mult (C_0 (M)).$ The previous proofs  go through if we replace
Proposition \ref{prop:phiForHyperbolicCase} above by the following more general Proposition \ref{prop:phiForSymmetricCase}, below.
We may now assume that $\Gamma$ is amenable, acts by oriented isometries on $\Mbar$ and fixes a point in the collar manifold  $N.$

\begin{proposition} Let $M$ denote a simply connected warped cone of nowhere positive curvature, with boundary $N.$  Assume that $\Gamma$ is amenable, acts by oriented isometries on $\Mbar,$ and fixes a point in the  boundary, $N.$
  There exists a $*$-homomorphism $\phi\colon C(N) \to \Mult (C_0 (M)),$ which is approximately equivariant in the sense
  that $\alpha_g(\phi (f))-\phi (\alpha_g(f))\in C_0(M).$   \label{prop:phiForSymmetricCase}  \end{proposition}
\begin{proof}
  By the Cartan-Hadamard theorem \cite[Thm. 3.1]{doCarmo}, our  space of nonpositive curvature is equipped with a
  globally defined exponential map, providing well-behaved geodesics radiating outwards from any chosen point. The geodesics extend uniquely to the closure $\Mbar.$

    \begin{floatingfigure}[r]{0.36\textwidth}{ 
      \begin{adjustbox}{width=0.36\textwidth}{  
          \CurvedTrianglePic } \end{adjustbox}  
      \captionof{figure}{Two geodesics of $\Mbar.$} \label{fig:geodesics} }\end{floatingfigure} 

  Consider, thus, the family of geodesics, in $M,$ to  points $n_i$ on the boundary $N,$  emanating from  the limit point $\idealpoint.$ Two such geodesics with distinct values of $n_i$ intersect only at $\idealpoint.$
      Let us now extend the domain of the given function $f\in C(N)$ to all of $\Mbar\setminus\idealpoint$ by defining it to be constant on all such geodesics. In terms of Figure \ref{fig:geodesics}, we propagate the values of $f$ upwards along geodesics, obtaining a function that is well-behaved except possibly at the  point \idealpoint.
  The function we so obtain is bounded and continuous at all regular points of the space $\Mbar,$ and so we can use the function as a multiplier of elements of $C_0(M).$ The remainder of the proof is just as in Proposition \ref{prop:phiForHyperbolicCase}.
\end{proof}

Exactly as in the previous section, we deduce:

\begin{corollary}[\bf The boundary of Dirac is Dirac]  \label{cor:The boundary of Dirac is Dirac.II}
Let
  \begin{eqnarray*}
    [D_M]&:=&\left(L^2(M)\oplus L^2(M),m ,\left(\begin{array}{cc}
        0   & D_M \\
        D_M & 0\end{array}
    \right)\right)              \in KK(C_0(M)\tensor C_1 ,\complex),\\{}
   \\{}
  &&\mbox{and}\\{}
  [D_{N}]   &:= &\left(L^2(N)\oplus L^2(N),m,\left(\begin{array}{ccc}
            0      & D_N & \\
            D_N & 0 & \end{array}
    \right)\right)              \in KK(C(N),\complex),\\{}\end{eqnarray*}
    where $m$ is the natural action by multiplication.
  The equivariant $K$-homology cycle $[D_M]$ is mapped to $[D_N]$ under the $K$-homology map induced by the equivariant extension $ 0\to C_0(M)\to \Cv\to C(N)\to 0.$
\end{corollary}

\subsection{Application to an Anghel-type theorem}
We now turn to the topic of an Anghel-type theorem, on a warped cone of nowhere positive curvature with an amenable discrete group acting quasi-parabolically by oriented isometries.
We first recall a classic lemma from the Hilbert space setting. Here, the potential $V$ is a self-adjoint unbounded multiplier acting on
(the sections of) the vector bundle $E.$
\begin{lemma} Let $D_M$ be a Dirac-type operator on a spin bundle $E$ over the warped cone $M.$
  Let $V\colon \Gamma(E)\to \Gamma(E)$  go to infinity at infinity in the warped cone. Then $D_M \pm iV,$ regarded as an
  unbounded operator on the Hilbert space $\L{2}(M,E),$ has compact resolvent.
  \label{lem:resolvent.compact.non-equivariantly}\end{lemma}
For a proof of the above lemma, see any of \cite{shortproof,anghel,anghel.fredholm,anghel.fredholm2}.
The above lemma was a key point in showing \cite{shortproof}  that $D_M+iV$ defines a (non-equivariant)
cycle in $KK(\C,\C).$ As shown using the Banach--Steinhaus theorem in Appendix A,\label{eq:Banach.Steinhaus}  if $D_m +iV$ does define an equivariant cycle   as in \cref{def:Gadmissible.approx}, then it follows that $\alpha_\gamma (D_M - iV) - (D_M -iV)$ is
bounded for each $\gamma\in\Gamma.$ 
So the potential $V$ must satisfy the strong but necessary assumption that for each $\gamma$, the difference
$\alpha_\gamma (V) - V$ is bounded.  But, in fact
a physically similar condition already appears in the classic Anghel's theorem, see page \pageref{th:original.Anghel's}, where it was assumed that the gradient of the potential, $[D,V],$ was bounded. Thus the assumption made is plausible.
Lemma \ref{lem:resolvent.compact.non-equivariantly} together with the above discussion implies that the necessary conditions for an unbounded equivariant cycle (\cref{def:Gadmissible})  hold:
\begin{lemma} Let $D_M$ be a Dirac-type operator on a spin bundle $E$ over the warped cone $M.$
  Let $V\colon \Gamma(E)\to \Gamma(E)$ go to infinity at infinity in the warped cone, and let $\alpha_\gamma (V) - V$ be
  bounded for each group element $\gamma\in\Gamma.$ Then the cycle
  $$\left(\L{2}(M,E)\oplus \L{2}(M,E),1,
    \left(\begin{array}{ccc}
        0      & D_M-iV & \\
        D_M+iV & 0\end{array}
    \right)\right)$$ is an equivariant $\KKg (\C,\C)$ cycle. \label{lem:resolvent.compact.equivariantly}\end{lemma}

Now, let us recall that:
\begin{enumerate}\item\ The Dirac operators over $N$ and over $M$ give
  $\KKg$-classes $[D_N]\in \KKg^0(C(N),\complex)$ and $[D_M]\in \KKg^1(C_0
    (M),\complex),$
  \item\ The geometrical properties of $M$ give an approximately equivariant $\KKg$-class
  $[r]\in \KKg^1(C(N),C_0(M)),$ as defined in \cref{lem:c0pairing.equiv},
  \item\ A given endomorphism $f$ that defines a $K$-theory class $[f]\in \KKg^0(\complex,C(N)),$ and, by Proposition \ref{prop:pairing.of.r.and.a.function} / \ref{prop:map.on.K1}, a $\KKg$ class  $[V]\in\KKg^1(\complex
    ,C_0(M)),$ where $V=f+ir,$ and
  \item\ The index of $D_M +iV$ defines a $\KKg$-class $[D_M +iV]\in \KKg(\complex
    ,\complex),$  see Lemma \ref{lem:resolvent.compact.equivariantly}.
\end{enumerate}

Then it follows, from the Kasparov product factorizations in Corollary \ref{cor:The boundary of Dirac is Dirac}, Proposition \ref{prop:pairing.of.r.and.a.function} / \ref{prop:map.on.K1}, and associativity of the Kasparov product that
\begin{eqnarray*}
  \Ind (D_M + iV):=[D_M + iV]&=&[V]\kasparovproduct_{C_0(M)}[D_M]\\
                                  &=&[f]\kasparovproduct_{C(N)}[r]\kasparovproduct_{C_0(M)}[D_M]\\
                                  &=&[f]\kasparovproduct_{C(N)}[D_N] .
                                \end{eqnarray*}
This shows that:
\begin{proposition} Let $M$ be a warped cone of nowhere positive curvature.
  Let $N$ denote its collar. Let $\Gamma$ be an amenable discrete group acting by oriented isometries on
  $0\to C(N)\to C_v (M) \to C_0(M)\to0,$ with a common fixed point in $N.$ Let $D_N$ and $D_M$ denote Dirac operators
  on $N$ and $M$ respectively, and let $f$ denote a potential on $N.$ Then
  $[D_M+i(r+\phi(f))]\in\KKg(\C,\C)$ factorizes equivariantly
  as $[D_N]\kasparovproduct_{C(N)} [f],$ an equivariant K-theory and K-homology pairing over the compact manifold $N.$ \end{proposition}

The above, then, is the basic form of an Anghel-type theorem that takes into account the restrictions
of an equivariant situation. Just as in the non-equivariant case, one can extend it to an apparently larger class of potentials, namely, the case of any $[D_M+iV]$ that is Kasparov homotopic to one of the above form. 

We have used hypotheses that are simple to state rather than aiming for maximal generality.
With regard to possible generalizations,
we should point out that the only nontrivial information from  Riemannian geometry that was used is:
\begin{itemize} \item Existence of Dirac operators, and
  \item The visibility space property of Hadamard manifolds (\emph{i.e.} the Cartan-Hadamard theorem).
\end{itemize}

The hypothesis of amenability  gives a pleasant proof of \cref{prop:equivariantL2} but can there be replaced by a more lattice-theoretical hypothesis: the existence of a $\Gamma$-invariant measure; as already pointed out in (the non-amenable case) of \cref{prop:equivariantL2}. As in  \cref{def:Gadmissible}, one could in the remainder of the work proceed by using cycles at the level of $\Psi_\Gamma(A,B).$

\appendix \section{Appendix: On strongly $\Gamma$-continuous unbounded regular operators}
\label{sect:appendix}

As a convenience to the reader,  we reproduce a hard-to-find known result in this appendix.

\begin{proposition}[\protect{\cite[Prop. 8.2]{kucerovsky.thesis}}] Let  $T$ be a regular operator on
a countably generated Hilbert module $
E$ and let $\Gamma$ be a locally compact Hausdorff group acting on $E$ by automorphisms. The following are equivalent:\begin{enumerate}
                                   \item  for each $e\in E$, the map \(g\mapsto\left(T-\alpha_g(T)\right)e\)
      is continuous as a map from $\Gamma$ into $E$; and
                                   \item the function  $g\mapsto \left(T-\alpha_g(T)\right)$ is in ${\mathcal L}(C(K,E))$ for every compact subset $K\subseteq \Gamma.$
                                 \end{enumerate}\label{pr:locally.bounded}
\end{proposition}
\begin{proof} It is clear that $ii)$ implies $i).$
 For the other direction, we need
to prove uniform boundedness over the compact set $K,$ which we do by using a prim\ae\!val form of the Banach--Steinhaus theorem:
\begin{theorem}[Banach--Steinhaus] Suppose that $P$ is a collection of continuous
functions $\{p_{\lambda}:X\longrightarrow {\mathbb R}\}_{\lambda\in\Lambda},$ and that $X$ is of second category.
Then if each cross-section $\lambda\longrightarrow p_{\lambda}(x)$ is bounded, there is an
nonempty open set $B_0\subset X$ with $p_{\lambda}(x)$ uniformly bounded for all
$\lambda\in\Lambda$ and $x\in B_0.$
\end{theorem}\ The above theorem has a standard proof, in
consequence of the definition of second category.
Now we return to the proof of the Proposition. Let
$L_g$ denote the given operator $T-\alpha_g(T).$
We apply the Banach--Steinhaus theorem to the
collection of functions $\{p_y:\Gamma\longrightarrow {\mathbb R}\}_{y\in E
\setminus \{0\}}$ defined by
$p_y(g)=\lVert L_gy\rVert/\lVert y\rVert.$ These functions are continuous because of
the hypothesis $i),$ and the cross--sections over $y\in E\setminus \{
0\},$ with
$g\in \Gamma$ fixed, are bounded by $\|L_g\|.$ A locally
compact Hausdorff space is a Baire space, so we conclude that there is an open neighborhood
$B_0\subset \Gamma$ with $\|L_g\|$ uniformly bounded on $B_0.$

We show that every point in $\Gamma$
will have a neighborhood upon which $\|L_g\|$ is uniformly bounded. Every
point in $\Gamma$ has a neighborhood given by a left translate $kB_0$ of
$B_0,$ and
\[\lVert L_{kb} \rVert=\lVert\alpha_{kb}(T)-T\rVert=\lVert\alpha_k(L_b-L_{k^{-1}})\rVert.\]
Hence $L_{kb}$ is uniformly bounded for all $b\in B_0,$ and we complete the proof by covering a given compact set $K$ with finitely many sets of the form $kB_0.$ \end{proof}
\begin{corollary} \label{def:Gadmissible.approx.locally.bounded.form}The set of unbounded approximately
  equivariant Fredholm
  modules $\Psi_\Gamma(A,B)$ is given by triples $(E,\phi ,D)$ where $E$ is a (graded)
  Hilbert $B$--module with $\Gamma$-action; $\phi :A\to\LE$ is a
  $*$--homomorphism; and $D$ is an
  unbounded regular degree one self--adjoint operator on $E$, such
  that
  \begin{enumerate}\item   the function $T-\alpha_g(T)$ is in ${\mathcal L}(C(K,E))$ for every compact subset $K\subseteq \Gamma;$
      \item the operator $(i+D)^{-1}$ is in $J_{\phi}$;
    \item the homomorphism $\phi$ satisfies $\alpha_g(\phi (a))-\phi (\alpha_g(a
          ))\in\KE,$ and
      \item for all $a$ in some dense subalgebra of $A,$ the
      commutator $[D,\phi (a)]$ is bounded on the domain of $D$.
  \end{enumerate}  \end{corollary}

\end{document}